\documentclass[a4paper,10pt]{amsart}
\usepackage{amsmath,amssymb,url,comment, enumerate}

\newcommand{\Z}{\mathbb{Z}}

\newcommand{\Q}{\mathbb{Q}}
\newcommand{\Hy}{\mathbb{H}}
\newcommand{\A}{\mathbb{A}}

\newcommand{\sn}[1]{\theta(O^+(L_{#1}))}
\newcommand{\sq}[1]{\dot{\mathbb Q}_{#1}^2}
\newcommand{\un}[1]{\mathbb Z_{#1}^{\times}}
\newcommand{\norms}[1]{N_{#1}(-n\Delta)}

\newtheorem{theorem}{Theorem}[section]
\newtheorem{lemma}[theorem]{Lemma}
\newtheorem{corollary}[theorem]{Corollary}
\newtheorem{proposition}[theorem]{Proposition}

\theoremstyle{remark}
\newtheorem*{remark}{Remark}

\numberwithin{equation}{section}

\newcommand{\qf}[1]{\langle #1 \rangle}
\newcommand{\gen}{\text{gen}\mspace{1mu}L}
\newcommand{\spn}{\text{spn}\mspace{1mu}L}

\newcommand{\ord}{\operatorname{ord}}

\newcommand{\repd}{\rightarrow}

\newcommand{\nrepd}{\not \rightarrow}

\newcommand{\sqf}[1]{\text{sqf}(#1)}

\title[Exceptional sets for spinor regular ternaries]{Exceptional sets for spinor regular ternary quadratic forms}

\author[A. G. Earnest]{A .G. Earnest}
\address{Department of Mathematics, Southern Illinois University, Carbondale, IL, 62901, U.S.A.}
\email{aearnest@siu.edu}


\begin{document}

\begin{abstract}

The goal of this note is to provide an analysis of the positive integers that are represented everywhere locally, but not globally, by each of the 29 spinor regular ternary quadratic forms that are not regular.

\end{abstract}

\maketitle

\section{Introduction}
It is well-known that there is no local-global principle for the representation of integers by integral quadratic forms. Consequently, it is of interest to identify the exceptional set of integers that are represented everywhere locally, but not globally, by such a form. For ternary forms this set can be either finite or infinite; when it is empty, Dickson called the form regular \cite{D}. That is, the regular forms are those that represent all integers represented by their genus. In terminology later introduced by Benham, Hsia, Hung and the author \cite{BEHH}, a ternary form is said to be spinor regular if it represents all the integers represented by its spinor genus.

An integer that is represented by a genus of ternary integral quadratic forms but not by every spinor genus in that genus is referred to as a spinor exceptional integer for the genus. In this note, the general theory of spinor exceptional integers will be used to determine the exceptional sets for each of the spinor regular positive definite ternary primitive integral quadratic forms for which the exceptional set is nonempty. In light of a result of the author and Haensch \cite{EHa}, it is known that there are exactly 29 inequivalent forms with this property. For the 27 of these forms that are alone in their spinor genus, explicit formulas for the numbers of representations of all positive integers by each such form were recently obtained by Aygin, Doyle, M\"{u}nkel, Pehlivan and Williams \cite{A}.

For each of the 29 forms, the exceptional set will be seen to consist of the integers lying in one or more squareclasses whose prime divisors satisfy certain congruence conditions. This phenomenon was observed by Jones and Pall \cite{JP} in their study of the forms in the genera of regular diagonal ternary forms. For example, they observed that the form \[4x^2+9y^2+9z^2+4xy+4xz+2yz\]represents all those integers represented by its genus except for the odd integers $m^2$ for which every prime factor $p$ of $m$ satisfies $p\equiv 1\,(\text{mod}\,4)$.  In all, Jones and Pall listed seven forms with similar properties. A full explanation for this type of behavior later emerged through spinor genus theory. Schulze-Pillot \cite{SP80} proved that each of these seven forms fails to represent precisely those integers that are spinor exceptional integers for their genus. The present paper can be viewed as an extension of the work of Schulze-Pillot applying the theory of spinor exceptional integers to complete the determination of the exceptional sets for each of the remaining spinor regular ternaries that are not regular.

We will follow the notations describing the spinor regular ternary forms in \cite{A}. In that paper, the forms are labelled as A1--A13, B1--B12 and C1--C4, where the forms are grouped according to the prime factors of their discriminant. In this paper, the main results for the three groups appear in Propositions 4.1, 5.1 and 6.1, respectively. 

The remainder of the paper is organized as follows. Section 2 contains notations and conventions that will be used throughout the paper. Some pertinent facts from the general theory of spinor exceptional integers will be reviewed in Section 3. The statements given there are special cases of general results from \cite{SP80}, stated here only in the generality needed to analyze the genera of the spinor regular ternary forms. On the topic of spinor exceptional integers, the interested reader may also wish to consult the excellent survey \cite{SP00}. Results for the forms A1--A13, B1--B12 and C1--C4 are given in Sections 4, 5 and 6, respectively. In the final section of the paper, the spinor regular forms that are not alone in their spinor genus, namely B4 and B11, are analyzed in more detail. Since the results on the positive integers not represented by these forms are obtained only for even integers in \cite{A}, we give the full statements here as Propositions 7.4 and 7.5.

\section{Ternary quadratic forms and lattices}

A ternary quadratic form will be described by the sextuple of integers appearing as its coefficients, where $(a,b,c,d,e,f)$ denotes the quadratic form \[F=ax^2+by^2+cz^2+dyz+exz+fxy.\] This form is referred to as classic if $d$, $e$ and $f$ are even, and non-classic otherwise. The discriminant $\Delta$ of $F$ is defined to be $\frac{1}{2}\textrm{det}M_F$, where $M_F$ is the matrix of second partial derivatives of the form.

For referencing the literature on spinor exceptional integers, it will be convenient to freely switch between the language of quadratic forms and lattices. For quadratic lattices, we will follow the terminology and notation of O'Meara's book \cite{OM}. To the ternary form $F$ given above, we associate the ternary lattice $L$ having Gram matrix $\frac{1}{2}M_F$ with respect to some basis. If $dL$ denotes the discriminant of this lattice, in the sense of \cite{OM}, then $\Delta = 4dL$. Let $\gen$ and $\spn$ denote the genus and spinor genus of $L$, respectively. For an integer $n$, we will use the notation $n\repd N$ to indicate that $n$ is represented by $N$, where $N$ can be the lattice $L$, its $p$-adic completion $L_p$ with respect to some prime $p$, or its genus $\gen$ or spinor genus $\spn$. Note that $n\repd L$ if and only if $n$ is represented by the original form $F$, and $n\repd \gen$ if and only if $n\repd L_p$ for all primes $p$.

\section{Determination of spinor exceptional integers}

Throughout this section, $L$ will denote an arbitrary positive definite integral ternary quadratic $\Z$-lattice, $\Delta$ will be the positive integer $4dL$, and $n$ will be a positive integer represented by $\gen$.

\subsection{General criteria for spinor exceptional integers}

In light of the results of \cite{SP80}, the spinor exceptional integers for a genus are determined by essentially local information. For a prime $p$, we denote the $p$-adic numbers and $p$-adic integers by $\Q_p$ and $\Z_p$, respectively, and the group of units of $\Z_p$ by $\un{p}$. The order of an element $\lambda$ of $\dot{\Q}_p$ will be denoted by $\ord_p(\lambda)$; so $\lambda = p^{\ord_p(\lambda)}\lambda_0$, with $\lambda_0\in \un{p}$.

Let $\sn{p}$ denote the group of spinor norms of rotations on $L_p$, $\theta(L_p,n)$ the relative spinor norm group defined in \cite[Definition 1, p. 531]{SP80}, and $N_p(-n\Delta)$ the group of local norms at $p$ of $\Q(\sqrt{-\sqf{n\Delta}})$, where $\sqf{\gamma}$ denotes the squarefree part of the positive integer $\gamma$. Concretely, \[N_p(-n\Delta)=\{\gamma \in \dot{\Q}_p:(\gamma,-n\Delta)_p=+1\},\] where $(\cdot,\cdot)_p$ denotes the Hilbert symbol at $p$. Note that $N_p(-n\Delta)$ is unchanged when the argument is changed by a square factor; in particular, $N_p(-n\Delta) = N_p(-\sqf{n\Delta})$.

\begin{lemma}\label{local square} If $-n\Delta \in \sq{p}$, then $\theta(L_p,n)=N_p(-n\Delta)$. \end{lemma}

\begin{proof} It follows from the definition of $\theta(L_p,n)$ that \[N_p(-n\Delta) \subseteq \theta(L_p,n) \subseteq \dot{\Q}_p.\] The result is immediate from this. \end{proof}

\begin{lemma}\label{containing units} If $\un{p}\sq{p} \subseteq N_p(-n\Delta)$, then $\ord_p(n\Delta)$ is even. Moreover, the reverse implication is true when $p$ is odd. \end{lemma}

\begin{proof} Suppose first that $\ord_p(-n\Delta)$ is odd. Then $(\varepsilon_p,-n\Delta)_p=(\varepsilon_p,p)_p=-1$ by \cite[63:11a]{OM}, where $\varepsilon_p$ denotes a unit of $\Z_p$ of quadratic defect $4\Z_p$. This establishes the first implication. If $p$ is odd and $\ord_p(n\Delta)$ is even, then $-n\Delta \in\un{p}\sq{p}$. Hence $(\gamma, -n\Delta)_p=+1$ for any $\gamma \in \un{p}\sq{p}$ by \cite[63:12]{OM}. \end{proof}

The following theorem of Schulze-Pillot \cite[Satz 2]{SP80} gives complete criteria for $n$ to be a spinor exceptional integer for $\gen$.

\begin{theorem}\label{SP criteria} The integer $n$ is a spinor exceptional integer for $\gen$ if and only if
\[\sn{p}\subseteq \norms{p} \quad \textrm{and} \quad \theta(L_p,n)= \norms{p}\]
for all primes $p$.
\end{theorem}

The local conditions appearing in the statement of this theorem will be discussed in more detail in the following two subsections. 

\smallskip

\subsection{Primes not dividing $2\Delta$}

Throughout this subsection, $p$ will denote a prime not dividing $2\Delta$. Consequently, $p$ is odd, $\ord_p(\Delta)=0$ and $L_p$ is unimodular. So $\sn{p}=\un{p}\sq{p}$ by \cite[92:5]{OM}.

\begin{lemma}\label{even order} $\sn{p}\subseteq N_p(-n\Delta) \iff \ord_p(n)\, \textrm{is even}$. \end{lemma}

\begin{proof} This follows from Lemma~\ref{containing units} since $\ord_p(n\Delta)=\ord_p(n)$. \end{proof}

\begin{lemma}\label{relative spinor norms} Assume that $\ord_p(n)$ is even.
\begin{itemize}
    \item[(i)] If $\ord_p(n)=0$, then $\theta(L_p,n)=N_p(-n\Delta)$.
    \item[(ii)] If $\ord_p(n)>0$, then \[\theta(L_p,n)=N_p(-n\Delta) \iff -n\Delta \in \sq{p}.\]
\end{itemize}
\end{lemma}

\begin{proof} If $-n\Delta \in \sq{p}$, then $\theta(L_p,n)=N_p(-n\Delta)$ by Lemma~\ref{local square}. So suppose that $-n\Delta \notin \sq{p}$. Here $\ord_p(n)$ even implies that $\sn{p}\subseteq N_p(-n\Delta)$ by Lemma~\ref{even order}. Since $p\nmid \Delta$ and $\ord_p(n)$ is even, $p$ is unramified in $\Q(\sqrt{-\sqf{n\Delta}})$. So \cite[Satz 3(a)]{SP80} applies with $r=s=0$. Hence, $\sn{p}\subseteq N_p(-n\Delta)$ holds if and only if $\ord_p(n)=0$. \end{proof}

For a positive integer $t$, let $M_t$ denote the multiplicative semigroup generated by $1$ and the set of all primes $p$ such that $-t\in \sq{p}$. For any $s\in \dot{\Q}$, note that $M_t=M_{s^2t}$; in particular, $M_t=M_{\sqf{t}}$. Also, $M_t^2$ will denote $\{w^2:w\in M_t\}$. 

\begin{corollary}\label{role of Mt} Assume that $\ord_p(n)$ is even for all $p\nmid 2\Delta$. Write $n=kw^2$, where all prime divisors of $k$ divide $2\Delta$ and $\textrm{g.c.d.}(w,2\Delta)=1$. Then \[\theta(L_p,n)=N_p(-n\Delta)\,\,\textrm{for all}\,\, p\nmid 2\Delta \iff w\in M_{n\Delta}.\] \end{corollary}

\subsection{Primes dividing $2\Delta$}

Throughout this subsection, $p$ will denote a prime divisor of $2\Delta$. 

\begin{lemma}\label{odd divisors} Let $p$ be an odd prime such that $-n\Delta \notin \sq{p}$ and $\sn{p}\subseteq N_p(-n\Delta)$.
 \begin{itemize}
    \item[(i)] If $L_p\cong \qf{1,p,p^2}$ or $L_p\cong \qf{1,p^2,p^3}$, then \[\theta(L_p,n)=N_p(-n\Delta) \iff \ord_p(n)\leq 1.\]
    \item[(ii)] If $L_p\cong \qf{1,p,p^3}$, then \[\theta(L_p,n)=N_p(-n\Delta) \iff \ord_p(n)\leq 2.\]
 \end{itemize}
\end{lemma}

\begin{proof} By the assumptions, we are in the situation of Satz 3(b) of \cite{SP80}. For $L_p\cong \qf{1,p,p^2}$ or $L_p\cong \qf{1,p, p^3}$, the stated result follows from subcase (ii) of Satz 3(b) with $r=1, s=2 \, \textrm{or} \, 3$. For $L_p\cong \qf{1,p^2,p^3}$, it follows from subcase (i) with $r=2, s=3$.
\end{proof}

For the prime $2$, there are numerous possibilities for the nature of a Jordan splitting of $L_2$. So in this case, we will state only the general form of the needed result, and provide specific references to the relevant subcases of the original result in \cite[Satz 4]{SP80} as they arise.

\begin{lemma}\label{prime 2} Assume that $-n\Delta \notin \sq{2}$ and $\sn{2}\subseteq N_2(-n\Delta)$. Then there exists a nonnegative integer $\lambda$ such that  \[\theta(L_2,n)=N_2(-n\Delta) \iff \ord_2(n)\leq \lambda.\]
\end{lemma}

\begin{corollary}\label{order 0} Let $p$ be a prime dividing $2\Delta$ and assume that $\lambda \geq 1$ if $p=2$. If $\sn{p}\subseteq \norms{p}$ and $\ord_p(n)=0$, then $\theta(L_p,n)=\norms{p}$. \end{corollary}

\begin{proof} For odd $p$, this follows either from Lemma~\ref{local square} or from Satz 3(b) of \cite{SP80}. For $p=2$, it follows either from Lemma~\ref{local square} or Lemma~\ref{prime 2}. \end{proof}

\section{Discriminants divisible only by 2}

The spinor regular ternaries with discriminants divisible only by 2 are labelled as A1--A13 in \cite{A}. Table~\ref{Group A} contains a list of representatives for all of the classes of forms in the genus of each of these forms, along with their separation into spinor genara and their discriminant. These representatives and their separation into spinor genera were obtained by the use of Magma \cite{Bo}. In each case, the spinor regular form is alone in its spinor genus, which is labelled Spinor Genus I, and the remaining classes constitute a second spinor genus labelled Spinor Genus II. These spinor genera will subsequently be referred to simply as SGI and SGII.

\begin{table}[h]
    \begin{center}
        \begin{tabular}{| c | l | l | c |}
            \hline
            \# & Spinor Genus I & Spinor Genus II & $\Delta$ \\ \hline \hline
            A1 & (2,2,5,2,2,0) & (1,1,16,0,0,0) & $2^6$ \\ \hline
            A2 & (1,4,9,4,0,0) & (1,1,32,0,0,0) & $2^7$ \\
               &               & (2,2,9,-2,2,0) &  \\ \hline
            A3 & (2,5,8,4,0,2) & (1,1,64,0,0,0) & $2^8$ \\
               &               & (2,2,17,2,-2,0) & \\
               &               & (1,4,17,-4,0,0) & \\ \hline
            A4 & (4,4,5,0,4,0) & (1,4,16,0,0,0) & $2^8$ \\ \hline
            A5 & (4,9,9,2,4,4) & (1,16,16,0,0,0) & $2^{10}$ \\ \hline
            A6 & (4,5,13,2,0,0) & (1,16,20,-16,0,0) & $2^{10}$ \\
               &               & (4,5,17,2,-4,-4) & \\ \hline
            A7 & (5,8,8,0,4,4) & (1,4,64,0,0,0) & $2^{10}$  \\
               &               & (4,4,17,0,4,0) & \\ \hline
            A8 & (4,8,17,0,4,0) & (1,8,64,0,0,0) & $2^{11}$ \\ \hline
            A9 & (9,9,16,8,8,2) & (1,16,64,0,0,0) & $2^{12}$ \\
               &               & (4,16,17,0,-4,0) & \\ \hline
            A10 & (4,9,32,0,0,4) & (1,32,32,0,0,0) & $2^{12}$  \\
               &                 & (4,17,17,2,4,4) & \\ \hline
            A11 & (5,13,16,0,0,2) & (4,16,21,16,4,0) & $2^{12}$ \\
                &                 & (4,5,64,0,0,-4) & \\ \hline
            A12 & (9,17,32,-8,8,6) & (1,16,256,0,0,0) & $2^{14}$ \\
                &                  & (16,16,17,-8,0,0) &  \\
                &                  & (4,16,65,0,4,0) &  \\ \hline
            A13 & (9,16,36,16,4,8) & (1,64,64,0,0,0) & $2^{14}$ \\
                &                  & (4,33,33,2,4,4) &  \\
                &                  & (4,17,64,0,0,-4) &  \\
                \hline
        \end{tabular}
    \end{center}
    \caption{Genera containing spinor regular ternaries with 2-power discriminant} \label{Group A}
    \end{table}

For each of the forms A1--A13, the data needed to determine the candidates for spinor exceptional integers for their genus is given in Table~\ref{Data for Group A}. The second column of the table gives the splitting of $L_2$ for the ternary lattice $L$ corresponding to the form. These splittings can be obtained by routine calculation using one of the representative forms in the genus. In all cases, the $2$-adic splitting is of the type \[\qf{b_1,b_22^r,b_32^s},\] where $b_i\in \un{2}$ and $r \leq s$ are nonnegative integers. The third column contains the local spinor norm group $\sn{2}$ for each lattice $L$. For the lattices for which the splitting of $L_2$ has $0<r<s$ (that is, A4, A6, A7, A8, A9, A11 and A12), the corresponding spinor norm groups $\sn{2}$ can be determined by Propositions 1.4, 1.6, 1.7 and 1.8 and Theorem 2.7 of \cite{EHs}. For the remainder of the cases, $L_2$ is split by a multiple of $\qf{1,1}$. Here $\theta(O^+(\qf{1,1}) = \{\gamma \in \dot{Q_2} : (\gamma, -1)_2=+1\} = \{1,2,5,10\}\sq{2}$ by Proposition B of \cite{H} \footnote{Throughout the paper, $\{a_1,\ldots,a_k\}\dot{F}^2$ will be used to denote $a_1\dot{F}^2 \cup\cdots \cup a_k\dot{F}^2$.}. Then $\sn{2}=\{1,2,5,10\}\sq{2}$ follows from Theorem 3.14(iv) of \cite{EHs}. The value of $\lambda$ in Lemma~\ref{prime 2} can be determined by one of the subcases of \cite[Satz 4]{SP80}. For each form, the specific subcase that applies is identified in the fourth column of the table, and the value of $\lambda$ appears in the last column.

\begin{table}[h]
    \begin{center}
        \begin{tabular}{| c | c | c | l | c |}
        \hline
        \# & $L_2$ & $\sn{2}$ & subcase & $\lambda$  \\ \hline \hline
        A1 & $\qf{1,1,2^4}$ & $\{1,2,5,10\}\sq{2}$ & (b)(iii) & 1 \\ \hline
        A2 & $\qf{1,1,2^5}$ & $\{1,2,5,10\}\sq{2}$ & (b)(iii) & 2 \\ \hline
        A3 & $\qf{1,1,2^6}$ & $\{1,2,5,10\}\sq{2}$ & (b)(iii) & 3  \\ \hline
        A4 & $\qf{1,2^2,2^4}$ & $\{1,5\}\sq{2}$ & (b)(iii) & 1  \\ \hline
        A5 & $\qf{1,2^4,2^4}$ & $\{1,2,5,10\}\sq{2}$ & (b)(i) & 1 \\ \hline
        A6 & $\qf{5,2^2,5\cdot 2^6}$ & $\{1,5\}\sq{2}$ & (b)(ii) & 1 \\ \hline
        A7 & $\qf{1,2^2,2^6}$ & $\{1,5\}\sq{2}$ & (b)(iii) & 3 \\ \hline
        A8 & $\qf{1,2^3,2^6}$ & $\{1,2,3,6\}\sq{2}$ & (c)(iii) & 1 \\ \hline
        A9 & $\qf{1,2^4,2^6}$ & $\{1,5\}\sq{2}$ & (b)(iii) & 3  \\ \hline
        A10 & $\qf{1,2^5,2^5}$ & $\{1,2,5,10\}\sq{2}$ & (b)(iv) & 1 \\ \hline
        A11 & $\qf{5,2^4,5\cdot 2^6}$ & $\{1,5\}\sq{2}$ & (b)(ii) & 3 \\ \hline
        A12 & $\qf{1,2^4,2^8}$ & $\{1,5\}\sq{2}$ & (b)(iii) & 5   \\ \hline
        A13 & $\qf{1,2^6,2^6}$ & $\{1,2,5,10\}\sq{2}$ & (b)(i) & 3  \\ \hline

        \end{tabular}
    \end{center}
    \caption{Data for A1--A13}\label{Data for Group A}
\end{table}

\begin{proposition} Let $f$ be one of the forms A1--A13, and let $n$ be a positive integer. Then $n$ is represented everywhere locally, but not globally, by $f$ if and only if $n$ lies in:
\begin{itemize}
    \item[(i)]$M_1^2$ for A1, A4, A5, A6, A10;
    \item[(ii)]$2M_1^2$ for A2;
    \item[(iii)]$M_1^2,4M_1^2$ for A3, A7, A9, A13;
    \item[(iv)]$M_2^2$ for A8;
    \item[(v)]$4M_1^2$ for A11;
    \item[(vi)]$M_1^2, 4M_1^2, 16M_1^2$ for A12.
\end{itemize}
\end{proposition}

\begin{proof} First consider the lattice $L$ corresponding to the form A1. In this case, we seek to show that \[n\nrepd L \iff n\in M_1^2.\] We first prove the forward implication. From $n\nrepd L$ it follows that $n\nrepd \spn$, since $L$ is spinor regular. So $n$ is a spinor exceptional integer for $\gen$. So by Theorem~\ref{SP criteria} and Lemma~\ref{even order} we see that $\ord_p(n)$ is even for all odd $p$; thus, \[n=2^tw^2,\] with $w$ odd.  As $\Delta=2^6$, we then have $\sqf{n\Delta} = 1$ or $2$, depending upon whether $t$ is even or odd, respectively. Since $\sn{2}\subseteq \norms{2}$ and $5\in \sn{2}\smallsetminus N_2(-2)$, it must be that $t$ is even and $M_{n\Delta}=M_1$. By Lemma~\ref{relative spinor norms} it follows that $w\in M_1$. Since $-1\notin \sq{2}$, it follows from Lemma~\ref{prime 2} that \[t=\ord_2(n)\leq \lambda = 1.\] Since $t$ is even, this implies that $t=0$. Hence, $n\in M_1^2$ as claimed.

For the reverse implication, assume that $n\in M_1^2$. So $n=w^2$, $w\in M_1$ and $\sqf{n\Delta}=1$. Note that $w$ is odd since $2\notin M_1$, as $-1\notin \sq{2}$. For $p$ odd, $\ord_p(n)$ even implies that $\sn{p}\subseteq \norms{p}$, by Lemma~\ref{even order}, and $\theta(L_p, n) = \norms{p}$ by Lemma~\ref{role of Mt}. Also $\sn{2} \subseteq N_2(-1)$. Since $\ord_2(n)=0$, it follows from Lemma~\ref{prime 2} that $\theta(L_2,n)=N_2(-n\Delta)$. By Theorem~\ref{SP criteria}, $n$ is a spinor exceptional integer for $\gen$. The form $\qf{1,1,16}$ lies in SGII for the genus of A1. So $1\repd SGII$ and hence $n\repd SGII$ since $n$ is a square. But then it must be that $n\nrepd SGI$ and so $n\nrepd L$. This completes the proof for A1.

The proofs for the forms A4, A5, A6 and A10 are identical. For A3, A7, A9, A11 and A13, the proof remains the same except that $\lambda =3$ for these cases. Thus the conditions for $\theta(L_2,n)=N_2(-n\Delta)$ in Lemma~\ref{prime 2} hold also with $t=2$, and integers of the type $4M_1^2$ are also spinor exceptional integers for these genera. Note that for A11, $1\nrepd \gen$, so the integers in $M_1^2$ need not be listed among the spinor exceptional integers. In this one instance, $4\repd SGII$, and so all integers of the type $4M_1^2$ are represented by SGII, and hence not by $L$. For the form A12, the proof is again unchanged except that in this case $\lambda =5$; thus the conditions for $\theta(L_2,n)=N_2(-n\Delta)$ in Lemma~\ref{prime 2} hold also with $t=4$, and the integers of the type $M_1^2, 4M_1^2$ and $16M_4^2$ are spinor exceptional integers for the genus.

This leaves two remaining cases, A2 and A8. The only difference for the proof of A2 is that $\ord_2(\Delta)$ is odd. Since the other aspects of the proof remain the same, including that $\sqf{n\Delta}=1$, we see in this case that $n = 2^tw^2$, where $t$ and $w$ are odd integers. As $\lambda = 1$, it must be that $t=1$ and so $n=2w^2$. Here $2\repd SGII$, so we conclude that no elements of the type $2M_1^2$ are represented by $L$.

Finally, in the case of A8, we have $\sn{2}=\{1,2,3,6\}\sq{2}$. So $\sn{p}\subseteq \norms{p}$ holds for all $p$ if and only if $\sqf{n\Delta}=2$. Hence $\sqf{n}=1$  for any spinor exceptional integer $n$ for this genus, since and $\ord_2(\Delta)=11$ is odd. As $\lambda =1$, the spinor exceptional integers for the genus must be of the form $w^2$, for some odd integer $w$. In this case, the prime divisors $p$ of $w$ must satisfy the condition that $-2\in \sq{p}$. Hence, $w\in M_2$, as asserted. \end{proof}

\begin{remark} The set $M_1$ appearing in the statement of the proposition is generated by 1 and the primes congruent to 1 modulo 4, and the set $M_2$ is generated by 1 and the primes congruent to 1 or 3 modulo 8. \end{remark}

\section{Discriminants divisible by 2 and 3}

The spinor regular ternaries with discriminants divisible by both 2 and 3 are labelled as B1--B12 in \cite{A}. Table~\ref{Group B} contains a list of representatives for all of the classes of forms in the genus of each of these forms, along with their separation into spinor genara and their discriminant. In each case, as for the previous grouping, the genus splits into two spinor genera, which are labelled as before. The spinor regular forms B4 and B11 are the first forms listed in SGI for their genus.

\begin{table}[h]
    \begin{center}
        \begin{tabular}{| c | l | l | c |}
            \hline
            \# & Spinor Genus I & Spinor Genus II & $\Delta$ \\ \hline \hline
            B1 & (3,3,4,0,0,3) & (1,1,36,0,0,1) & $2^23^3$ \\ \hline
            B2 & (3,4,4,4,3,3) & (1,3,10,-3,1,0) & $2^23^3$ \\ \hline
            B3 & (1,7,12,0,0,1) & (3,3,13,-3,3,-3) & $2^23^4$  \\
                &               & (1,1,108,0,0,1) &  \\ \hline
            B4 & (3,7,7,5,3,3) & (1,1,144,0,0,1) & $2^43^3$ \\
               & (3,3,16,0,0,-3) & (1,3,37,3,1,0) &  \\ \hline
            B5 & (4,4,9,0,0,4) & (1,12,12,12,0,0) & $2^43^3$  \\ \hline
            B6 & (3,4,9,0,0,0) & (1,3,36,0,0,0) & $2^43^3$ \\ \hline
            B7 & (4,9,12,0,0,0) & (1,12,36,0,0,0) & $2^63^3$  \\ \hline
            B8 & (4,9,28,0,4,0) & (1,36,36,-36,0,0) & $2^43^5$  \\
                &               & (9,13,13,-10,-6,-6) &  \\ \hline
            B9 & (9,16,16,16,0,0) & (1,48,48,-48,0,0) & $2^83^3$ \\ \hline
            B10 & (13,13,16,-8,8,10) & (4,13,37,-2,4,-4) & $2^83^3$ \\ \hline
            B11 & (9,16,48,0,0,0) & (1,48,144,0,0,0) & $2^{10}3^3$ \\
                & (16,25,25,-14,16,-16) & (4,49,49,-46,4,4)& \\ \hline
            B12 & (9,16,112,16,0,0) & (1,144,144,144,0,0) & $2^83^5$ \\
                &                   & (9,49,49,-46,6,6) & \\ \hline
        \end{tabular}
    \end{center}
    \caption{Genera containing spinor regular ternaries with discriminant divisible by 2 and 3}\label{Group B}
\end{table}

The data necessary to determine the spinor exceptional integers for these genera is summarized in Table~\ref{Data for Group B}. We first record the local splittings at the primes 2 and 3 for the corresponding lattices. The splittings for $L_3$ are given in the second column of Table~\ref{Data for Group B}. In every case, it follows from \cite[Satz 3]{K56} that \[\sn{3}=\{1,3\}\sq{3}.\]The $2$-adic splittings appear in the third column. Some additional explanation is in order here. This grouping contains both classic and non-classic forms. For the non-classic forms B1--B4, the splitting given in the table is for the lattice $L_2'$ obtained by scaling $L_2$ by $2$. For the remaining forms B5--B12, which are classic, the splitting of $L_2'=L_2$ is given. Since the spinor norm group $\sn{2}$ is unaffected by scaling, this makes it possible to directly apply the results of \cite{EHs} to determine $\sn{2}$ from $L_2'$ in all cases. We use the notations $\Hy$ and $\A$ to denote the binary $\Z_2$-lattices with Gram matrices $\left(\begin{smallmatrix} 0 & 1 \\1 & 0\end{smallmatrix}\right)$ and $\left(\begin{smallmatrix} 2 & 1 \\1 & 2\end{smallmatrix}\right)$, respectively. If $M$ is one of these matrices and $\alpha\in\Z$, $\alpha M$ will denote the matrix resulting from multiplying each entry of $M$ by $\alpha$. In all cases occurring in Table~\ref{Data for Group B}, $L_2'$ has a binary Jordan component either of odd order or of even order (see \cite[Definition 3.1, p. 79]{EHs} for an explanation of this terminology), and the complementary component has the same order. It follows from Theorem 3.14(i) of \cite{EHs} that in all cases \[\sn{2}=\un{2}\sq{2}.\] The value of $\lambda$ in Lemma~\ref{prime 2} can be determined by one of the subcases of \cite[Satz 4(a)]{SP80}. For each form, the specific subcase that applies is identified in the fourth column of the table, and the value of $\lambda$ appears in the last column.

\begin{table}[h]
    \begin{center}
        \begin{tabular}{| c | c | c | l | c |}
        \hline
        \# & $L_3$ & $L'_2$  & subcase & $\lambda$  \\ \hline \hline
         B1 & $\qf{1,3,3^2}$ & $\A \perp \qf{2^3}$ & (ii)($\beta$) & 1 \\ \hline
         B2 & $\qf{1,3,3^2}$  & $\Hy \perp \qf{5\cdot 2^3}$ &  (ii)($\alpha$) & 1 \\ \hline
         B3 & $\qf{1,3,3^3}$ & $\A \perp \qf{3\cdot 2^3}$ & (ii)($\beta$) & 1 \\ \hline
         B4 & $\qf{1,3,3^2}$ & $\A \perp \qf{2^5}$ & (ii)($\beta$) & 3 \\ \hline
         B5 & $\qf{1,3,3^2}$ & $\qf{1}\perp 2\A$ & (ii)($\gamma$) & 1 \\ \hline
         B6 & $\qf{1,3,3^2}$ & $\qf{1,3,2^2}$ & (i)($\beta$) & 1 \\ \hline
         B7 & $\qf{1,3,3^2}$ & $\qf{1,2^2,3\cdot 2^2}$ & (i)($\alpha$) & 1 \\ \hline
         B8 & $\qf{1,3^2,3^3}$ & $\qf{1}\perp 2\A$ & (ii)($\gamma$) & 1 \\ \hline
         B9 & $\qf{1,3,3^2}$  & $\qf{1}\perp 2^3\A$ & (ii)($\gamma$) & 3 \\ \hline
         B10 & $\qf{1,3,3^2}$ & $\qf{5}\perp 2^3\Hy$ & (ii)($\gamma$) & 3 \\ \hline
         B11 & $\qf{1,3,3^2}$ & $\qf{1,2^4,3\cdot 2^4}$ & (i)($\alpha$) & 3 \\ \hline
         B12 & $\qf{1,3^2,3^3}$ & $\qf{1}\perp 2^3\A$ & (ii)($\gamma$) & 3 \\ \hline

        \end{tabular}
    \end{center}
    \caption{Data for B1--B12}\label{Data for Group B}
\end{table}

\begin{proposition} Let $f$ be one of the forms B1--B12, and let $n$ be a positive integer. Then $n$ is represented everywhere locally, but not globally, by $f$ if and only if $n$ lies in:
\begin{itemize}
    \item[(i)]$M_3^2$ for B1, B2, B5, B6, B7, B8;
    \item[(ii)]$3M_3^2$ for B3;
    \item[(iii)]$M_3^2,4M_3^2$ for B4, B9, B11, B12;
    \item[(iv)]$4M_3^2$ for B10.
\end{itemize}
\end{proposition}

\begin{proof} First consider the lattice $L$ corresponding to the form B1. Assume first that $n\repd \gen$ but $n\nrepd L$. Then $n \nrepd \spn$, since $L$ is spinor regular. So $n$ is a spinor exceptional integer for $\gen$. From Theorem~\ref{SP criteria}, Lemma~\ref{even order}, the computation of $\sn{2}$, and Lemma~\ref{containing units}, it follows that $\ord_p(n)$ is even for all $p\neq 3$. So $\sqf{n\Delta}=1$ or $3$. Since $\sn{3} \not\subseteq N_3(-1)$, it must be that $\sqf{n\Delta}=3$. So $n=2^t3^sw^2$ with $s,t$ even and $w\in M_3$, by Lemma~\ref{role of Mt}. By Lemma~\ref{odd divisors} and Lemma~\ref{prime 2}, $\ord_p(n)\leq 1$ for $p=2,3$. Hence, $s=t=0$ and the conclusion follows.

For the reverse implication, assume that $n=w^2$ with $w\in M_3$. The lattice corresponding to $(1,1,36,0,0,1)$ lies in SGII and represents $1$, and so $n\repd SGII$; in particular, $n\repd \gen$. Since $\Delta=2^43^3$, we have $\sqf{n\Delta}=3$; thus, $\norms{p}=N_p(-3)$ and $M_{n\Delta}=M_3$. Since $\ord_p(n\Delta)$ is even for $p\neq 2,3$, it follows from Lemma~\ref{even order} that $\sn{p}\subseteq \norms{p}$, and, by Corollary~\ref{role of Mt}, $\theta(L_p,n)=\norms{p}$ for all $p\neq 2,3$. For the primes $2$ and $3$, the direct computations show that $\sn{2} = \un{2}\sq{2}=N_2(-3)=\norms{2}$ and $\sn{3}=\{1,3\}\sq{3}=N_3(-3)=\norms{3}$. It then follows by Corollary~\ref{order 0} that $\theta(L_p,n)=\norms{p}$. So the criteria of Theorem~\ref{SP criteria} are met and $n$ is a spinor exceptional integer for $\gen$. So $n\nrepd SGI$ and, in particular, $n\nrepd L$. This completes the proof for B1.

The proofs for the forms B2, B5, B6, B7 and B8 proceed in exactly the same way. For B4, B9, B11 and B12, the proof is analogous except that $\lambda =3$ for these cases. Thus the conditions for $\theta(L_2,n)=N_2(-n\Delta)$ in Lemma~\ref{prime 2} hold also with $t=2$, and integers of the type $4M_3^2$ are also spinor exceptional integers for these genera.

This leaves two remaining cases, B3 and B10. The only difference for the proof of B3 is that $\ord_3(\Delta)$ is even. Since the other aspects of the proof remain the same, including that $\sqf{n\Delta}=1$, we see in this case that if $n$ is a spinor exceptional integer for $\gen$, then $n = 2^t3^sw^2$, where $t$ is even, $s$ is odd, and $w\in M_3$. As $\lambda = 1$, it must be that $t=0$, and it follows from Lemma~\ref{odd divisors} that $s=1$; so $n=3w^2$, $w\in M_3$. In this case, $3\repd SGII$, so we conclude that no elements of the type $3M_3^2$ are represented by $L$.

For the case B10, the proof of the forward implication proceeds as before up to the point where we conclude that $n=2^tw^2$ with $t$ even and $w\in M_3$. In this case, $L_2\cong \qf{5}\perp 2^3\A$. So the assumption that $n\repd L_2$ implies that $t \neq 0$. Hence, $n\in 4M_3^2$, as claimed. For the reverse implication, observe that $4\repd SGII$, and so $4M_3^2 \repd SGII$. The remainder of the argument then proceeds as before. \end{proof}

\begin{remark} The set $M_3$ appearing in the statement of the proposition is generated by 1 and the primes congruent to 1 modulo 3. \end{remark}

\section{Discriminants divisible by 2 and 7}

The spinor regular ternaries with discriminants divisible by 2 and 7 are labelled as C1--C4 in \cite{A}. Table~\ref{Group C} contains a list of representatives for all of the classes of forms in the genus of each of these forms, along with their separation into spinor genara and their discriminant. In each case, as for the previous two groupings, the genus splits into two spinor genera, which are labelled as before.

\begin{table}[h]
    \begin{center}
        \begin{tabular}{| c | l | l | c | c |}
            \hline
            \# & Spinor Genus I & Spinor Genus II & $\Delta$ & $L'_2$ \\ \hline \hline
            C1 & (2,7,8,7,1,0) & (1,7,14,7,0,0) & $2^27^3$ & $\Hy \perp \qf{2}$ \\
               &               & (1,2,49,0,0,1) & & \\ \hline
            C2 & (7,8,9,6,7,0) & (4,7,15,-7,4,0) & $2^47^3$ & $\A\perp \qf{5\cdot 2^3}$ \\
               &               & (1,7,51,-7,-1,0) & & \\ \hline
            C3 & (8,9,25,2,4,8) & (1,28,56,-28,0,0) & $2^67^3$ & $\qf{1}\perp 2\Hy$ \\
               &                & (4,8,49,0,0,4) & & \\ \hline
            C4 & (29,32,36,32,12,24) & (4,29,197,-2,-4,4) & $2^{10}7^3$ & $\qf{5}\perp 2^3\A$ \\
               &                     & (16,32,53,-8,16,-16) & & \\ \hline
        \end{tabular}
    \end{center}
    \caption{Genera containing spinor regular ternaries with discriminant divisible by 2 and 7}\label{Group C}
\end{table}

For each lattice $L$ corresponding to one of the forms C1--C4, we have $L_7\cong \qf{1,7,7^2}$, and it follows from \cite[Satz 3]{K56} that \[\sn{7}=\{1,7\}\sq{7}=N_7(-7).\] The $2$-adic splittings for these lattices are given in the last column of Table~\ref{Group C}. The forms C1 and C2 are non-classic and the splittings listed for them are for the lattice $L'_2$ obtained from $L_2$ by scaling by 2; for C3 and C4, which are classic, the splittings listed are for $L'_2=L_2$. In all cases, it follows from \cite[Theorem 3.14(i)]{EHs} that \[\sn{2} = \un{2}\sq{2}.\]

\begin{proposition} Let $f$ be one of the forms C1--C4, and let $n$ be a positive integer. Then $n$ is represented everywhere locally, but not globally, by $f$ if and only if $n$ lies in:
\begin{itemize}
    \item[(i)] $M_7^2$ for C1, C2, C3;
    \item[(ii)] $4M_7^2$ for C4.
\end{itemize}
\end{proposition}

\begin{proof} Let $L$ be a lattice corresponding to one of the forms C1-C3. Assume first that $n\repd \gen$ but $n\nrepd L$. As before, $n$ is a spinor exceptional integer for $\gen$. From Theorem~\ref{SP criteria}, Lemma~\ref{even order} and the computation of $\sn{2}$, it follows that $\ord_p(n\Delta)$ is even for all $p\neq 7$. So $\sqf{n\Delta}=1$ or $7$. Since $\sn{7} \not\subseteq N_7(-1)$, it must be that $\sqf{n\Delta}=7$. So $n=2^s7^tw^2$ with $s,t$ even and $w\in M_7$, by Lemma~\ref{role of Mt}. By Lemma~\ref{odd divisors}, $\ord_7(n)\leq 1$ and so $t=0$. Since $2\in M_7$, it follows that $n\in M_7^2$ as claimed.

For the reverse implication, assume that $n=w^2$ with $w\in M_7$. It can be seen from the representatives of SGII that $1\repd SGII$, and so $n\repd SGII$; in particular, $n\repd \gen$. Since $\ord_2(\Delta)$ is even and $\ord_7(\Delta)$ is odd, we have $\sqf{n\Delta}=7$; thus, $\norms{p}=N_p(-7)$ for all $p$, and $M_{n\Delta}=M_7$. Since $\ord_p(n\Delta)$ is even for $p\neq 2,7$, it follows from Lemma~\ref{even order} that $\sn{p}\subseteq \norms{p}$, and from Corollary~\ref{role of Mt}, that $\theta(L_p,n)=\norms{p}$ for all $p\neq 2,7$. Also, $\sn{7}=N_7(-7)=\norms{7}$ and $\ord_7(n)=0$ gives $\theta(L_7,n)=\norms{7}$, by Corollary~\ref{order 0}. Since $-7\in\sq{2}$, we have $\sn{2}=\un{2}\sq{2}\subseteq N_2(-7)=\norms{2}$, and, by Lemma~\ref{local square}, $\theta(L_7,n)=\norms{2}$. So the criteria of Theorem~\ref{SP criteria} are met and $n$ is a spinor exceptional integer for $\gen$. So $n\nrepd SGI$ and, in particular, $n\nrepd L$. This completes the proof for C1-C3.

For the case C4, the proof of the forward implication proceeds as above to the point where we conclude that $n=2^sw^2$ with $s$ even and $w\in M_7$. In this case, $L_2\cong \qf{5}\perp 2^3\A$. So the assumption that $n\repd L_2$ implies that $s\neq 0$. Hence, $n\in 4M_7^2$, as claimed. For the reverse implication, observe that $4\repd SGII$, and so $4M_7^2 \repd SGII$. The remainder of the argument then proceeds as before.    \end{proof}

\begin{remark} The set $M_7$ appearing in the statement of the proposition is generated by 1 and the primes congruent to 1, 2 or 4 modulo 7. \end{remark}

\section{Forms not alone in their spinor genus}

For each of the 27 spinor regular ternaries that are alone in their spinor genus, Aygin et al \cite{A} summarize in Table A.17 the list of all positive integers not represented by the form. As detailed in that paper, the results for several of these forms appeared earlier in work of Lomadze \cite{L} and Berkovich \cite{B}. However, for the remaining two forms which are not alone in their spinor genus, the method of \cite{A} yields only the even integers that fail to be represented. For completeness, we will state here the full results for these two forms, which are B4 and B11 in the list. Before stating the results, we will establish three lemmas needed to analyze the local obstructions to representation by these forms. 

\begin{lemma}\label{2-adic lemma1} Let $L$ be a ternary quadratic $\Z$-lattice such that $L_2\cong M\perp\qf{16}$, where $M$ is a lattice corresponding to $x^2+xy+y^2$, and let $n$ be a positive integer. Then $n\nrepd L_2$ if and only if $n=2+4\ell$ or $8+16\ell$, for some nonnegative integer $\ell$. \end{lemma}

\begin{proof} Since $\un{2}\repd M$, it follows that $\un{2},4\un{2}\repd L_2$. Also, for $x,y \in \Z_2$, $x^2+xy+y^2$ lies in either $\un{2}$ or $4\Z_2$. So all odd integers and integers of the type $4n_0$ with $n_0$ odd are represented by $L_2$, and no integer of the type $2+4\ell$ can be represented by $L_2$.

Consider $n=4n_0$ with $n_0$ even. If there exist $x,y\in \Z_2$ such that \begin{equation}\label{1st eq} n=x^2+xy+y^2+16z^2,\end{equation} then $x,y\in 2\Z_2$ and the righthand side of Equation~\eqref{1st eq} is in $4\Z_2$. Hence, no element of the type $8+16\ell$ is represented by $L_2$.

Finally, let $n=16n_0$, $n_0\in \Z$. Since $\un{2}\repd M$, either $n_0$ or $n_0-1$ is represented by $M$. So there exist $x_0,y_0,z\in \Z_2$ such that \begin{equation}\label{2nd eq}n_0=x_0^2+x_0y_0+y_0^2+z^2.\end{equation} It follows that Equation~\eqref{1st eq} is satisfied with $x=4x_0, y=4y_0$. This completes the proof. \end{proof}

\begin{lemma}\label{2-adic lemma 2} Let $L$ be a ternary quadratic $\Z$-lattice such that $L_2\cong \qf{1,16,48}$, and let $n$ be a positive integer. Then $n\nrepd L_2$ if and only if $n=5+8\ell$, $2+4\ell$, $3+4\ell$, $8+16\ell$, or $12+16\ell$, for some nonnegative integer $\ell$. \end{lemma}

\begin{proof} Here $n\repd L_2$ if and only if there exist $x,y,z\in \Z_2$ such that \begin{equation}\label{first eq} n=x^2+16y^2+48z^2.\end{equation} If $2 \nmid n$, then Equation~\eqref{first eq} is solvable if and only if $n\equiv 1\,(\text{mod}\,8)$, thus ruling out integers of the type $5+8\ell$ and $3+4\ell$. If $2\mid n$, then $x=2x_0$ for some $x_0\in \Z_2$, so \begin{equation}\label{second eq} n=4x_0^2+16y^2+48z^2.\end{equation} So it must be that $4\mid n$, ruling out integers of the type $2+4\ell$. Write $n=4n_0$, with $n_0\in \Z$. Dividing Equation~\eqref{second eq} through by 4 then gives \begin{equation}\label{third eq} n_0=x_0^2+4y^2+12z^2.\end{equation} If $2\nmid n_0$, then Equation~\eqref{third eq} is solvable if and only if $n_0\equiv 1\,(\text{mod}\,4)$. This rules out integers of the type $4(3+4\ell)=12+16\ell$. If $2\mid n_0$ then $x_0\in 2\Z_2$, and so the righthand side of Equation~\eqref{third eq} is in $4\Z_2$. This rules out integers of the type $8+16\ell$. Finally, the lattice $\qf{1,1,3}$ is isotropic over $\Z_2$ and is therefore $\Z_2$-universal by \cite[Proposition 4.1]{EG}. Hence, all positive integers divisible by 16 are represented by $L_2$. This completes the proof. \end{proof}

\begin{lemma}\label{3-adic lemma} Let $L$ be a ternary quadratic $\Z$-lattice such that $L_3\cong \qf{1,3,9}$, and let $n$ be a positive integer. Then $n\nrepd L_3$ if and only if $n=2+3\ell$ or $n=9^k(6+9\ell)$, for some nonnegative integers $k,\ell$. \end{lemma}

\begin{proof} Let $V$ denote the underlying quadratic space. By a computation of Hasse symbols, it follows that $V_3$ is anisotropic; hence $\alpha \nrepd V_3$ for any $\alpha \in -dV = -3\sq{3}$ (see, e.g., \cite[Lemma 2.2]{EG}). As $9^k(6+9\ell) = 3^{2k}\cdot 3 (2+3\ell) \in -3\sq{3}$, no such integer can be represented by $L_3$. Integers of the type $n=2+3\ell$ are ruled out for representation by $L_3$ by the Local Square Theorem \cite[63:1]{OM}.

It remains to show that all other integers are represented by $L_3$. Write $n=3^tn_0$, with $t$ a nonnegative integer and $n_0\equiv 1,2\,(\text{mod}\,3)$. If $n_0\equiv 1\,(\text{mod}\,3)$, then there exists $\lambda \in \un{3}$ such that $n_0=\lambda^2$. So if $t=0$ and $n_0\equiv 1\,(\text{mod}\,3)$ then $n=\lambda^2\repd L_3$. If $t=2k+1$ is odd and $n_0\equiv 1\,(\text{mod}\,3)$, then $n=3(3^k\lambda)^2\repd L_3$. By \cite[92:1b]{OM}, $\un{3}\repd \qf{1,1}$. So $3^2\un{3}\repd \qf{1,3^2}$. It then follows that $n\repd L_3$ whenever $t$ is even and $t\geq 2$. This exhausts all cases and completes the proof. \end{proof}

\begin{proposition}\label{B4} The form $3x^2+7y^2+7z^2+3xy+3xz+5yz$ represents  a positive integer $n$ if and only if $n$ is not of the type \[2+3\ell, 9^k(6+9\ell), 4^a(2+4\ell)\,\,\text{or}\,\, 4^aM_3^2,\] where $k,\ell$ are nonnegative integers and $a\in \{0,1\}$. \end{proposition}

\begin{proof} This is the form B4. Let $L$ be a lattice corresponding to this form and let $n$ be a positive integer. By \cite[92:1b]{OM}, $n\repd L_p$ for all primes $p\neq 2,3$. For $p=2$ and $p=3$, the conditions for representation by $L_p$ are given in Lemmas~\ref{2-adic lemma1} and \ref{3-adic lemma}, respectively. Consequently, $n\repd \gen$ if and only if $n$ is not of one of the types $2+3\ell$, $9^k(6+9\ell)$ or $4^a(2+4\ell)$ for some nonnegative integers $k,\ell$ and $a\in\{0,1\}$. If $n$ is not of one of these types, then $n\repd L$ if and only if $n$ does not lie in $M_3^2$ or $4M_3^2$, by Proposition 5.1(iii). \end{proof}

\begin{proposition}\label{B11} The form $9x^2+16y^2+48z^2$ represents  a positive integer $n$ if and only if $n$ is not of the type \[5+8\ell, 4^a(2+4\ell), 4^a(3+4\ell) \,\,\text{or}\,\, 4^aM_3^2,\] where $\ell$ is a nonnegative integer and $a\in \{0,1\}$. \end{proposition}

\begin{proof} This is the form B11. Let $L$ be a lattice corresponding to this form and let $n$ be a positive integer. By \cite[92:1b]{OM}, $n\repd L_p$ for all primes $p\neq 2,3$. For $p=2$ and $p=3$, the conditions for representation by $L_p$ are given in Lemmas~\ref{2-adic lemma 2} and \ref{3-adic lemma}, respectively. Consequently, $n\repd \gen$ if and only if $n$ is not of one of the types $5+8\ell$, $4^a(2+4\ell)$ or $4^a(3+4\ell)$ for some nonnegative integers $\ell$ and $a\in\{0,1\}$. If $n$ is not of one of these types, then $n\repd L$ if and only if $n$ does not lie in $M_3^2$ or $4M_3^2$, by Proposition 5.1(iii). \end{proof}



\begin{thebibliography}{0}


\bibitem{A}
    Z.S. Aygin, G. Doyle, F. M\"{u}nkel, L. Pehlivan and K.S. Williams,
    Representation numbers of spinor regular ternary quadratic forms,
    {\it Integers} {\bf 21} (2021), \#A99 (110 pages).

\bibitem{BEHH}
    J.W. Benham, A.G. Earnest, J.S. Hsia and D.C. Hung,
    Spinor regular ternary quadratic forms,
    {\it J. London Math. Soc.} {\bf 42} (1990), 1--10. 

\bibitem{B}
    A. Berkovich,
    On the Gauss E$\Upsilon$PHKA theorem and some allied inequalities,
    {\it J. Number Theory} {\bf 148} (2015), 1--18. 
    
\bibitem{Bo}
    W. Bosma, J. Cannon and C. Playoust,
    The Magma algebra system. I. The user language,
    {\it J. Symbolic Comput.} {\bf 24} (1997), 235--265. 
    
\bibitem{D}
    L.E. Dickson, 
    Ternary quadratic forms and congruences,
    {\it Ann. of Math.} {\bf 28} (1926/27), 333--341. 
    

\bibitem{EG}
    A.G. Earnest and G.L.K. Gunawardana,
    Local criteria for universal and primitively universal quadratic forms,
    {\it J. Number Theory} {\bf 225} (2021), 260–280.
    
\bibitem{EHa}
    A.G. Earnest and A.N. Haensch,
    Completeness of the list of spinor regular ternary quadratic forms,
    {\it Mathematika} {\bf 65} (2019), 213--235. 

\bibitem{EHs}
    A.G. Earnest and J.S. Hsia,
    Spinor norms of local integral rotations II,
    {\it Pacific J. Math.} {\bf 61} (1975), 71--86. 
    
\bibitem{H}
    J.S. Hsia,
    Spinor norms of local integral rotations I,
    {\it Pacific J. Math.} {\bf 57} (1975), 199--206. 

\bibitem{OM}
    O.T. O'Meara,
    {\it Introduction to Quadratic Forms},
    Springer-Verlag, Berlin, 1973.

\bibitem{JP}
    B.W. Jones and G. Pall,
    Regular and semi-regular positive ternary quadratic forms,
    {\it Acta Math.} {\bf 70} (1940), 165--191.

\bibitem{L}
    G.A. Lomadze,
    Formulas for the number of representations of numbers by certain regular and semiregular ternary quadratic forms belonging to two-class genera (Russian),
    {\it Acta Arith.} {\bf 34} (1977/78), 131--162. 

\bibitem{K56}
    M. Kneser,
    Klassenzahlen indefiniter quadratischer Formen in drei oder mehr Ver\"{a}nderlichen,
    {\it Arch. Math. (Basel)} {\bf 7} (1956), 323--332. 

\bibitem{SP80}
    R. Schulze-Pillot,
    Darstellung durch Spinorgeschlechter tern\"{a}rer quadratischer Formen,
    {\it J. Number Theory} {\bf 12} (1980), 529--540.

\bibitem{SP00}
    R. Schulze-Pillot,
    Exceptional integers for genera of integral ternary positive definite quadratic forms,
    {\it Duke Math. J.} {\bf 102} (2000), 351--357.


\end{thebibliography}
\end{document}